\documentclass[11pt]{amsart}

\usepackage{hyperref}

\makeatletter
\providecommand{\href}[2]{}
\def\tospace#1{\@tospace#1 \tospace@delimiter}
\def\@tospace#1 #2\tospace@delimiter{#1}
\def\MR#1{\edef\MR@help{{http://www.ams.org/mathscinet-getitem?mr=\tospace{#1}}{\tospace{#1}}}%
\expandafter\href\MR@help .}
\makeatother

\providecommand*{\backref}{}
\providecommand*{\backrefalt}{}
\renewcommand*{\backref}[1]{}
\renewcommand*{\backrefalt}[4]{%
    \ifcase #1 %
    \or
      Cited page #2.
    \else
      Cited pages #2.
    \fi
}

\newcommand{\BV}{\mathcal{BV}}

\newcommand{\dd}{\;{\rm d}}

\newcommand{\N}{\mathbb{N}}

\DeclareMathOperator{\dLeb}{dLeb}

\newcommand{\norm}[1]{\left\| #1 \right\|}

\DeclareMathOperator{\boL}{\mathcal{L}}
\DeclareMathOperator{\boM}{\mathcal{M}}
\DeclareMathOperator{\Lip}{Lip}

\newcommand{\coloneqq}{\mathrel{\mathop:}=}
\newcommand{\eqqcolon}{=\mathrel{\mathop:}}

\newtheorem*{thm}{Main theorem}
\newtheorem{lem}{Lemma}

\theoremstyle{definition}

\begin{document}

\title[Spectral gap on Lipschitz, not on $\BV$]{An interval map with a spectral gap on Lipschitz functions, but not on bounded variation functions}
\author{S\'{e}bastien Gou\"{e}zel}

\address{IRMAR,
Universit\'{e} de Rennes 1, 35042 Rennes, France}
\email{sebastien.gouezel@univ-rennes1.fr}

\date{September 2, 2008}
\begin{abstract}
We construct a uniformly expanding map of the interval,
preserving Lebesgue measure, such that the corresponding
transfer operator admits a spectral gap on the space of
Lipschitz functions, but does not act continuously on the space
of bounded variation functions.
\end{abstract}

\maketitle

To study the statistical properties of uniformly expanding
maps, a very successful strategy is to find a Banach space on
which the transfer operator of the map admits a spectral gap
(see e.g.~\cite{baladi_decay} and references therein). In
dimension $d\geq 2$, several Banach spaces have been
considered, suited to different situations. However, in one
dimension, it is commonly acknowledged that the space $\BV$ of
bounded variation functions is the best space to work with (see
nevertheless \cite{keller_GeneralVar} for examples of suitable
function spaces when the derivative is only H\"{o}lder continuous).

The purpose of this note is to exhibit a very specific example
showing that, quite surprisingly, the behavior on $\Lip$ can be
better than the behavior on $\BV$: we will construct a
uniformly expanding map $T$ of the interval $I=[0,1)$,
preserving Lebesgue measure, such that the transfer operator
$\boL$ associated to $T$ (defined by duality by $\int f\cdot
g\circ T \dLeb = \int \boL f \cdot g \dLeb$) admits a spectral
gap on the space $\Lip$ of Lipschitz functions, while it does
not act continuously on $\BV$.

The main motivation for this short note is the articles
\cite{durieu1} and \cite{durieu2}, in which the authors prove
results on empirical processes assuming that the transfer
operator admits a spectral gap on $\Lip$. Since similar results
were already known when the transfer operator admits a spectral
gap on $\BV$, our example shows that these articles really have
new applications.

Our example is a Markov map of the interval, with (infinitely
many) full branches. The difficulty in constructing such an
example is that the usual mechanisms that imply continuity on
$\Lip$ also imply continuity on $\BV$: if $\boL_v$ denotes the
part of the transfer operator corresponding to an inverse
branch $v$ of $T$ (so that $\boL=\sum_v \boL_v$), one has
$\norm{\boL_v}_{\BV\to \BV} \leq C \norm{\boL_v}_{\Lip\to
\Lip}$, for some universal constant $C$. To take this fact into
account, our example will therefore involve pairs of branches
$v$ and $w$ such that $\boL_v$ and $\boL_w$ behave badly on
$\Lip$, while the sum $\boL_v+\boL_w$ behaves nicely, thanks to
nontrivial compensations.

\medskip

If $f$ is a function defined on $[0,1)$, we will write
$\norm{f}_{\Lip}=\sup |f| + \Lip(f)$, where $\Lip(f)=\sup
|f(x)-f(y)|/|x-y|$.

Let us fix a sequence $(a_n)_{n\in \N^*}$  of positive real
numbers, with $\sum a_n <1/4$, as well as an integer $N>0$. Our
construction will depend on them, and we will show that a good
choice of these parameters gives the required properties.

Let $I_n=[ 4\sum_{k=1}^{n-1} a_k, 4\sum_{k=1}^n a_k)$. We
decompose this interval (of length $4 a_n$) into two
subintervals of length $2 a_n$ that we denote respectively by
$I_n^{(1)}$ and $I_n^{(2)}$. We have
  \begin{equation*}
  \int_0^1 a_n (1+2\cos^2(2 \pi n^4 x))\dd x=
  \int_0^1 a_n (1+2\sin^2(2 \pi n^4 x))\dd x=
  2 a_n.
  \end{equation*}
We can therefore define two maps $v_n$ and $w_n$ on $[0,1)$,
with respective images $I_n^{(1)}$ and $I_n^{(2)}$, such that
$v_n'(x)=a_n (1+2\cos^2(2 \pi n^4 x))$ and $w_n'(x)=a_n
(1+2\sin^2(2 \pi n^4 x))$. We define $T$ on $I_n$ by imposing
that $v_n$ and $w_n$ are two inverse branches of $T$. It
remains to define $T$ on $J\coloneqq [4\sum_{n=1}^\infty a_n,
1)$. We cut this interval into $N$ subintervals of equal
length, and let $T$ send each of these intervals affinely onto
$[0,1)$.

\begin{thm}
\label{mainthm} If $a_n=1/(100 n^3)$ and $N=4$, then the map
$T$ preserves Lebesgue measure and the associated transfer
operator has a spectral gap on the space of Lipschitz
functions, with a simple eigenvalue at $1$ and no other
eigenvalue of modulus $1$. On the other hand, the transfer
operator does not act continuously on the space of bounded
variation functions.
\end{thm}

For the proof, we will work with general values of $a_n$ and
$N$, and specialize them only at the end of the argument.

\begin{lem}
The map $T$ preserves Lebesgue measure.
\end{lem}
\begin{proof}
At a point $x$, the sum of the derivatives of the inverse branches of
$T$ is equal to
  \begin{align*}
  \sum_{n=1}^\infty (v_n'(x)+w_n'(x)) &+ N\cdot |J|/N
  \\&=\sum_{n=1}^\infty a_n (2+ 2\cos^2(2 \pi n^4 x) + 2\sin^2(2\pi
  n^4  x)) + |J|
  \\&= \sum_{n=1}^\infty 4 a_n +|J|
  =1.
  \qedhere
  \end{align*}
\end{proof}

Let $\boL$ be the transfer operator of $T$ associated to
Lebesgue measure, given by $\boL f(x)=\sum_{T(y)=x}
f(y)/T'(y)$. Let also
  \begin{equation*}
  \boL_n f(x)=\sum_{T(y)=x,\;y\in I_n} f(y)/T'(y)=v'_n(x)f(v_n
  x)+w'_n(x)f(w_n x),
  \end{equation*}
and $\boM f(x)=\sum_{T(y)=x,\;y\in J} f(y)/T'(y)$, so that
$\boL=\sum_n \boL_n + \boM$.

\begin{lem}
\label{lem:badBV} Assume that $\limsup n^4 a_n= \infty$. Then,
for any $k\geq 1$, the operator $\boL^k$ does not act
continuously on the space $\BV$ of bounded variation functions.
\end{lem}
\begin{proof}
Let us fix an inverse branch $v$ of $T$ with $v(I)\subset J$
($v$ is affine and its slope is $|J|/N$). Fix $k>0$. Let
$\chi_n$ be the characteristic function of the interval
$v^{k-1} \circ v_n(I)$. Its variation is bounded by $2$.
Moreover, $\boL^k \chi_n(x)=a_n (1+2 \cos^2(2 \pi n^4 x))
(|J|/N)^{k-1}$, hence the variation of $\boL^k \chi_n$ is at
least $C(k) a_n n^4$, for some $C(k)>0$. This concludes the
proof.
\end{proof}

The next lemma is the crucial lemma: it ensures that $\boL_n$
behaves well on $\Lip$, while we have seen (in the proof of
Lemma \ref{lem:badBV}) that it is not the case on $\BV$. We
should emphasize that each branch of $\boL_n$ behaves badly on
$\Lip$: this is the addition of the two branches that gives a
better behavior, thanks to the compensations of the bumps of
$v'_n$ and $w'_n$.

\begin{lem}
\label{lem_borne_Ln} We have $\norm{\boL_n f}_{\Lip} \leq a_n
(32\pi n^4 a_n + 8) \norm{f}_{\Lip}$.
\end{lem}
\begin{proof}
Let us decompose $\boL_n /a_n$ as
  \begin{align*}
  \boL_n f(x)/a_n &= (1+2\cos^2(2\pi n^4 x)) f(v_n x) +  (1+2 \sin^2(2 \pi n^4 x)) f(w_n x)
  \\ &= (f(v_n x)+3 f(w_n x)) + 2 \cos^2 (2\pi n^4 x) (f(v_n x)-f(w_n x))
  \\& \eqqcolon \boL_n^{(1)}f(x)+ 2 \boL_n^{(2)}f(x).
  \end{align*}

We have $\norm{\boL_n^{(1)}f}_{C^0}\leq 4 \norm{f}_{C^0}$.
Moreover, since $v_n$ and $w_n$ are contracting by a factor at
least $3 a_n$, we have $\Lip(\boL_n^{(1)}f) \leq 12 a_n
\Lip(f)\leq 3 \Lip(f)$. Hence, $\norm{\boL_n^{(1)}f}_{\Lip}
\leq 4 \norm{f}_{\Lip}$.

Let us now turn to $\boL_n^{(2)}$. It satisfies
$\norm{\boL_n^{(2)}f}_{C^0}\leq 2 \norm{f}_{C^0}$. Moreover,
for any $x,y$
  \begin{multline*}
  \boL_n^{(2)} f(x)-\boL_n^{(2)}f(y)
  =(\cos^2(2\pi n^4 x) -\cos^2(2\pi n^4 y)) (f(v_n x)-f(w_n x)) \\ + \cos^2(2\pi n^4 y) ((f(v_n x)-f(v_n y))- (f(w_n x)-f(w_n y))).
  \end{multline*}
Since the derivative of $\cos^2(2\pi n^4 x)$ is bounded by
$4\pi n^4$, the first term is at most $4\pi n^4  |x-y| |f(v_n
x)-f(w_n x)|$. The distance between $v_n x$ and $w_n x$ being
at most the length of $I_n$, i.e., $4 a_n$, we therefore get a
bound $16 \pi n^4 a_n |x-y|\Lip(f)$. For the second term,
$\cos^2(2\pi n^4 y)$ is bounded by $1$,  and $|f(v_n x)-f(v_n
y)|\leq 3 a_n \Lip(f)|x-y| \leq \Lip(f)|x-y|$. Similarly,
$|f(w_n x)-f(w_n y)|\leq \Lip(f)|x-y|$. This second term is
therefore bounded by $2 |x-y|\Lip(f)$. We obtain
  \begin{equation*}
  \Lip(\boL_n^{(2)} f) \leq (16\pi n^4 a_n + 2)\Lip(f).
  \end{equation*}
This proves the lemma.
\end{proof}

Let us fix once and for all $a_n=c/ n^3$, with $c$ small enough
so that
  \begin{equation}
  \label{eq:contracteLn}
  \sum_{n=1}^\infty \norm{\boL_n f}_{\Lip} \leq \frac{1}{2} \norm{f}_{\Lip}.
  \end{equation}
Using Lemma \ref{lem_borne_Ln}, one can check that $c=1/100$ is
sufficient.

\begin{lem}
\label{lemLY} If $N\geq 4$, the operator $\boL$ satisfies
  \begin{equation*}
  \norm{\boL f}_{\Lip} \leq \frac{3}{4} \norm{f}_{\Lip} +  \norm{f}_{C^0}.
  \end{equation*}
\end{lem}
\begin{proof}
The operator $\boM$ satisfies by construction $\norm{\boM
f}_{C^0} \leq |J| \norm{f}_{C^0}$, and $\Lip(\boM f) \leq |J|
\Lip(f) \cdot \frac{|J|}{N}$ (since every inverse branch of $T$
on $J$ contracts by a factor $|J|/N$). With
\eqref{eq:contracteLn}, we obtain
  \begin{equation*}
  \norm{\boL f}_{\Lip} \leq \left( \frac{1}{2}+ \frac{|J|^2}{N}\right) \norm{f}_{\Lip} + |J|  \norm{f}_{C^0}.
  \end{equation*}
This gives the desired conclusion if $N\geq 4$, since $|J|\leq
1$.
\end{proof}

\begin{proof}[Proof of the main theorem]
By Lemma \ref{lem:badBV}, the choice of $a_n$ ensures that the
transfer operator $\boL$ does not act continuously on $\BV$. On
the other hand, since the inclusion of $\Lip$ in $C^0$ is
compact, Lemma \ref{lemLY} and Hennion's theorem show that the
essential spectral radius of $\boL$ acting on $\Lip$ is $\leq
3/4<1$.

It remains to study the eigenvalues of modulus $1$ of $\boL$.
Let $f$ be a nonzero eigenfunction for such an eigenvalue
$\lambda$. Let $x$ be such that $|f(x)|$ is maximal, then
  \begin{multline*}
  |f(x)| = | \boL f (x)|= \left|\sum_{T(y)=x} f(y)/T'(y)\right|
  \\
  \leq \sum_{T(y)=x} |f(y)| / T'(y)
  \leq \sum_{T(y)=x} \sup |f| / T'(y)
  = |f(x)|.
  \end{multline*}
There is equality everywhere in these inequalities. Hence,
$|f(y)|=|f(x)|$ for any preimage $y$ of $x$, and the complex
numbers $f(y)$ all have the same argument. This shows that $f$
is constant on the set $T^{-1}(x)$. Applying the same argument
to $\boL^n$, we see that $f$ is constant on $T^{-n}(x)$. This
set being more and more dense as $n$ tends to infinity, this
shows that $f$ is constant, concluding the proof.
\end{proof}


\bibliography{biblio}
\bibliographystyle{amsalpha}

\end{document}